\documentclass[11pt]{amsart}
\usepackage{amssymb,amsfonts,graphicx}
\usepackage{amsmath,amsthm}
\usepackage{chngcntr}
\usepackage{apptools}
\usepackage[all]{xy}
\usepackage{enumerate}
\xyoption{all}
\usepackage{amscd}
\usepackage{color}

\newcommand{\lra}{{\longrightarrow }}
\newcommand{\cO}{{\mathcal O}}
\newcommand{\cM}{{\mathcal M}}

\newcommand{\cF}{{\mathcal F}}

\newcommand{\cA}{{\mathcal A}}
\newtheorem{theorem}[equation]{Theorem}
\newtheorem{corollary}[equation]{Corollary}
\newtheorem{proposition}[equation]{Proposition}
\newtheorem{lemma}[equation]{Lemma}

\newtheorem{thm}{Theorem}[section]
\newtheorem{cor}[thm]{Corollary}
\newtheorem{lem}[thm]{Lemma}
\newtheorem{prop}[thm]{Proposition}
\theoremstyle{definition}
\newtheorem{defn}[thm]{Definition}
\newtheorem{rem}[thm]{Remark}

\numberwithin{equation}{section}

\newcommand{\PP}{\mathbb P}

\newcommand{\ra}{\rightarrow}

\newcommand{\cP}{\mathcal{P}}
\newcommand{\cR}{\mathcal{R}}

\DeclareMathOperator{\Aut}{{Aut}}
\DeclareMathOperator{\Hom}{{Hom}}

\DeclareMathOperator{\Ker}{Ker}

\DeclareMathOperator{\Pic}{Pic}
\DeclareMathOperator{\Nm}{{Nm}}

\DeclareMathOperator{\Supp}{Supp}
\DeclareMathOperator{\Spec}{Spec}
\DeclareMathOperator{\Sym}{Sym}
\DeclareMathOperator{\Sing}{Sing}

\DeclareMathOperator{\rk}{rk}

\title[Generic injectivity of the Prym map]{Generic injectivity of the Prym map for double ramified coverings }
\author{ Juan Carlos Naranjo,  Angela Ortega \\  (with an appendix by Alessandro Verra)}

\address{J.C. Naranjo \\ Departament de Matem\`atiques i Inform\`atica \\ Universitat de Barcelona \\Spain }
\email{jcnaranjo@ub.edu}
              
\address{A. Ortega \\ Institut f\"ur Mathematik, Humboldt Universit\"at zu Berlin \\ Germany}
\email{ortega@math.hu-berlin.de}

\address{A. Verra \\ Dipartimento di Matematica, Universit\`a Roma Tre \\ Italy}
\email{verra@mat.uniroma3.it}

\thanks{The first author was partially supported by the Proyecto de Investigaci\'on MTM2015-65361-P}

\begin{document}

\begin{abstract} 
In this paper we consider  the Prym map for double coverings of curves of genus $g$ ramified at $r>0$ points. That is, the 
map associating to a double ramified covering its Prym variety.  The generic
Torelli theorem  states that the Prym map is generically injective as soon as the dimension of the space of coverings is less or 
equal to the dimension of the space of polarized abelian varieties. We prove the generic injectivity of the Prym map 
in the cases of double coverings of curves with:  (a)  $g=2$, $r=6$, and (b) $g= 5$, $r=2$.
In the first case the proof is constructive and  can be extended to the range $r\ge \max \{6,\frac 23(g+2) \}$.
For (b)  we study the fibre along the locus of the intermediate Jacobians of cubic threefolds to conclude the generic injectivity. 
This completes the work of Marcucci and Pirola who proved this theorem for all the other cases, except for
the bielliptic case $g=1$ (solved later by Marcucci and the first author), and the case  $g=3, r=4$ considered previously by Nagaraj and Ramanan, and also by Bardelli, Ciliberto and Verra where the degree of the map is $3$.

The paper closes with an appendix by Alessandro Verra with an independent result, the rationality of the moduli space of coverings with $g=2,r=6$, whose proof is 
self-contained. 
\end{abstract}

\maketitle

It is well known that a general Prym variety of dimension at least $6$  is the Prym variety of a unique \'etale covering. This is the 
so-called generic Torelli Theorem for Prym varieties. In a modular way this result can be reformulated as the generic injectivity of the
 Prym map 
\[
\mathcal R_g \lra \mathcal A_{g-1}
\]
sending an \'etale covering in $\mathcal R_g$ to its Prym variety which turns out to be a principally polarized abelian variety. In the
 last years, starting with the seminal work by Marcucci and Pirola (see \cite{mp}), the ramified case has attracted attention and 
the corresponding generic Torelli problem has shown to be a natural problem plenty of rich geometry. In the mentioned paper generic
 Torelli is proved for a big amount of cases using degeneration techniques and only a few special situations remained open. Our purpose
  is to prove the generic Torelli theorem in these last cases.

In order to establish our results more precisely let us define the basic objects that take part in the main statements.
Let $C$ be an irreducible smooth complex projective curve of genus $g$ and let $\pi: D \ra C $ be a smooth double covering ramified in $r>0$ points. The associated Prym variety to $\pi$ is defined as
$$
P(\pi):= \Ker \Nm_{\pi } \subset JD.
$$
This is an abelian subvariety of dimension $g-1+\frac{r}{2}$ with induced polarization $\Xi$ of type 
$$
\delta=(1,\ldots, 1, \underbrace{2, ...,2}_{g \textnormal{ times }} \ ).
$$
Given a divisor $B$ in $C$ of even degree $r>0$ with no multiple points  and  
a line bundle $\eta \in \Pic^{r/2}(C)$ with $B \in  |\eta^{\otimes 2}|$,   the projection 
$$
\pi: D = \boldsymbol{ \Spec } (\cO_C \oplus \eta^{-1}) \ra \boldsymbol{\Spec }\cO_C =C. 
$$
defines a double covering branched over $B$. Conversely, every double ramified covering of $C$ arises in this way. 
Hence these coverings are parametrized by the moduli space 
$$
\cR_{g,r}:=\{ (C, \eta,  B) \ \mid \ \eta \in \Pic^{\frac{r}{2}}(C), B \text{ reduced divisor in } |\eta^{\otimes 2}|  \}.
$$
Let $\cA^\delta_{g-1+\frac{r}{2}}$ denote the moduli space of abelian varieties of dimension  $g-1+\frac{r}{2}$ and polarization of 
type $\delta $.
For any $g\ge 1$ and $r>0$ we define the Prym map  by

\begin{eqnarray*}
\cP_{g,r} : \cR_{g,r} & \ra & \cA^\delta_{g-1+\frac{r}{2}}\\
 (\pi : C \ra D) & \mapsto & ( P (\pi) , \Xi).
\end{eqnarray*}

The codifferential of $\cP_{g,r}$ at a generic point $[(C, \eta, B)]$  is given by the multiplication map
$$
d\cP_{g,r}^* : \Sym^2H^0(C, \omega_C \otimes \eta) \ra H^0(C, \omega_C^2 \otimes \cO(B))
$$
which is known to be surjective (\cite{lo}), therefore $\cP$ is generically finite, if and only if 
$$
\dim \cR_{g,r} \leq \dim \cA^{\delta}_{g-1+\frac{r}{2}}.
$$
So, by an elementary count of parameters  we get that the Prym map is generically finite as soon as one of the following conditions holds: 
either $r\geq 6$ and $g\geq 1$, or $r=4$ and $g\geq 3$, or $r=2$ and $g\geq 5$. 
 
Marcucci and Pirola (\cite{mp}), and later Marcucci and the first author (\cite{mn}), have proved the generic injectivity in all the
 cases except three: 
\begin{enumerate}
 \item [(a)] $r=4, g=3$, 
 \item [(b)] $r=6, g=2$,
 \item [(c)] $r=2, g=5$.
\end{enumerate}
Case (a) was considered previously by Nagaraj and Ramanan, also by Bardelli, Ciliberto and Verra (see \cite{nr}, \cite{bcv}). 
They proved (in different ways) that  the degree of the Prym map is $3$ and is the only case where the map is generically finite but the generic Torelli Theorem fails.

This paper has two goals: to prove generic Torelli Theorem for the remaining cases (b) and (c), and to give a constructive proof of the generic 
Torelli Theorem for  $r\ge \max \{6, \frac 23 (g+2)\}$.  More precisely we prove (Theorem \ref{constructive_Torelli}):

\begin{thm}  \label{constructive_Torelli_intro}
  Let $(C, \eta , B)$ be a generic element in $\mathcal R_{g,r}$ and let $(P, \Xi)$ its Prym variety. Assume that  $r\ge \max \{6, \frac 23 (g+2)\}$.
  Then the element $(C, \eta, B)$ is uniquely determined by the base locus $Bs$ of the linear system $| \Xi |$.
\end{thm}

In particular, this establishes  the  generic injectivity of the Prym map $\cP_{2,6}: \mathcal R_{2,6} \ra \mathcal A_4^{(1,1,2,2)}$ (Corollary \ref{corolario-Torelli}).
The image of the Prym map gives a divisor in $\mathcal A_4^{(1,1,2,2)}$, which is invariant under the natural involution in  $\mathcal A_4^{(1,1,2,2)}$
(see \cite{ps}). This divisor could be useful to understand the nature of the birational geometry of $\mathcal A_4^{(1,1,2,2)}$, for instance to compute
its Kodaira dimension, which seems to be unknown.

Recall that in the case of \'etale double coverings there are two different constructive proofs  of the generic 
injectivity  of  the Prym map: one due to Welters, which uses degeneration methods and works for genus $g\geq 16$ and another one by Debarre, where
the covering is reconstructed from the tangent cones to the stable singularities of  theta divisor of $P$, this holds for curves of genus $g \geq 7$ (see \cite{d} and \cite{w}).
 In the same spirit but with a different method, we present in Theorem \ref{constructive_Torelli_intro} a procedure to reconstruct the covering 
for $r\ge \max \{6, \frac 23 (g+2)\}$. This gives another proof of Marcucci and Pirola results in this range. It also provides
a different proof for the case $r \geq 6$, $g=1$, which was proven in \cite{mn}.
The idea of the proof of Theorem \ref{constructive_Torelli_intro}  is to show that the base locus of the map 
$P \dashrightarrow |\Xi|^{\vee}$ is symmetric and invariant by the action of the kernel of the isogeny $\lambda_{\Xi} :P\lra P^{\vee}$. 
The quotient by this action turns out to be, under some hypothesis of genericity, birational to the symmetric product $ D^{(\frac r2 -1)}$. 
Moreover the symmetry of the base locus induces an involution on $D^{(\frac r2 -1)}$. Then an extended Torelli theorem proved by Martens allows us to conclude.

The second part of the paper is devoted to prove the generic injectivity in case $(c)$. The main result of Section 2 is the following (Theorem \ref{gen_injective}):

\begin{thm} \label{injective-52}
The Prym map $\cP: \cR_{5,2} \ra \cA_{5}$ is generically injective.
\end{thm}

The moduli space $\cR_{5,2}$ can be embedded into the Beauville partial compactification $\overline{\cR}_{6}$ of admissible  \'etale double coverings of genus 6 curves, by identifying the two branch points on the covering and on the base curve. The closure of the image of   $\cR_{5,2}$ is an irreducible divisor $\Delta^n$ in the boundary of $\overline{\cR}_{6}$ (denoted by 
$\Delta_0^{\mbox{ram}}$ in \cite{fgsv}).  Using the definition of the Prym map for admissible coverings one deduces from the main theorem in \cite{ds}
that the degree of  the Prym map  $\cP_6 :\cR_{6} \ra \cA_5 $ restricted to $\Delta^n$ is $1$, although the total degree of the map is  $27$. 
The class of the image of $\Delta^n$ under the Prym map in $\cA_5$ has been computed in \cite{fgsv} (Theorem 0.8) but that does not give the 
degree of the restriction map to $\Delta^n$.

For the proof of Theorem \ref{injective-52} we use the description of the fibres of the Prym map $\cP_6$  over the locus of intermediate Jacobians of cubic threefolds
(after a suitable blowup) intersected  with the boundary divisor $\Delta^n$, and exploit the geometry of the cubic threefolds.

\vskip 5mm
As a summary, the previous results together with the main theorems in \cite{mp}, \cite{mn} and \cite{nr} gives the following result:
\begin{thm}
 Let $\cP_{g,r}$ be a Prym map with $r>0$. We assume that 
 \[
 \dim \cR_{g,r} \le \dim \cA_{g-1+\frac r2}^\delta. 
 \]
 Then $\cP_{g,r}$ is generically injective except for $g=3, r=4$, in which case the degree is $3$.
\end{thm}

The paper closes with an appendix by Alessandro Verra whith an independent result, namely the rationality of the moduli space $\cR_{2,6}$, whose 
proof is  self-contained. 

{\bf Acknowledgements:} We are grateful to G.P. Pirola and Alessandro Verra for many helpful suggestions and stimulating discussions on the subject. We thank
Herbert Lange for reading and commenting on a preliminary version of the paper. 

\section{A constructive generic Torelli Theorem for $r\ge \max \{6, \frac 23 (g+2)\}$}

 The goal of the section is to prove the following Theorem:
 
 \begin{thm}\label{constructive_Torelli}
  Let $(C, \eta , B)$ be a generic element in $\mathcal R_{g,r}$ and let $(P, \Xi)$ its Prym variety. Assume that $r\ge \max \{6, \frac 23 (g+2)\}$.
  Then the element $(C, \eta, B)$ is uniquely determined by the base locus $Bs$ of the linear system $| \Xi |$.   \end{thm}

  \vskip 3mm
  We first explain in a few words how $Bs$ determines explicitly the covering $\pi: D\ra C$ attached to $(C,\eta, B)$: the kernel of the polarization map  
$\lambda_{\Xi }:P \ra P^{\vee }$ acts on $Bs$, which moreover is a symmetric subvariety of $P$. It turns out that the quotient $Bs/\Ker(\lambda_{\Xi})$ is 
birational to the symmetric product $D^{(\frac r2-2)}$ and inherits a natural involution $\sigma'$. By a theorem of Martens (see \cite{mar}) $\sigma '$ 
is induced 
by an involution $\sigma $ on $D$. Since $C=D/\langle \sigma \rangle$ the covering is determined. It is convenient to remark that Martens' proof is also 
constructible since it relies on the Boolean calculus of special subvarieties of Jacobians. An elegant cohomological proof of the same theorem was given 
by Ran in \cite{ran}.
  
  \vskip 3mm
    This Theorem reproves the main theorem in \cite{mp} under the hypothesis $r\ge \max \{6, \frac 23 (g+2)\}$ and recovers completely the result in \cite{mn}. 
Observe that our proof is constructive: we provide a procedure to recover the covering $\pi:D\ra C$ from its Prym variety. Instead, in both quoted papers, degeneration 
methods are used. It would be interesting to find also a constructive proof for the cases where the map is generically injective and $r< \frac 23 (g+2)$.
    
    
\begin{cor}\label{corolario-Torelli}
 The Prym map $\mathcal R_{2,6} \longrightarrow \mathcal A_4^{(1,1,2,2)}$ is generically injective.
\end{cor}

The rest of the section is devoted to  prove Theorem \ref{constructive_Torelli}.  As above we denote by $\pi:D \ra C$ the ramified double covering 
attached to $(C,\eta,B)$. Remember that the Prym variety $(P, \Xi)$ associated to the covering $\pi$ has polarization of type $(1,\ldots, 2, \ldots ,2)$, 
where $2$ appears $g$ times.

 By Riemann-Roch Theorem we have that $\dim H^0(P, \cO_P(\Xi)) = 2^g$.
Let $\varphi_{\Xi} : P \dashrightarrow \PP H^0(P, \cO_P(\Xi))^{\vee}$
be the  map defined by the polarization.  We fix a translation  $\Theta_D \subset \Pic^0(D)$ of the 
theta divisor of $JD$ and a translation of $\Theta_C$ such that $\Theta _D \cap JC=2\Theta_C$. Then we define the map
$$
\delta : P \dashrightarrow |2\Theta_C|,   \qquad  z \mapsto (\Theta_{D, z})_{|_{\pi^* JC}},
$$  
where $\Theta_{D, z}= \Theta_D + z$.   This is well-defined: indeed this is equivalent to the statement that
\[
 \left( \Theta_{D,z}-\Theta_D \right){|_{\pi^* JC}}
\]
 is linearly equivalent to the trivial divisor. In other words, that 
\begin{equation}\label{well-defined}
(\pi^*)^ {\vee}\circ \lambda _{\Theta_D} \circ i=0, 
\end{equation}
where $i$ is the embedding of $P$ in $JD$, and $\lambda _{\Theta_D}$ is the isomorphism
\[
 \lambda_{\Theta_D}: JD \lra \Pic^0(JD)
\]
sending $z$ to $\mathcal O_{JD}(\Theta_{D,z}-\Theta_D)$. To see that (\ref{well-defined}) holds we remind the following standard formula for coverings:
\[
 \Nm_{\pi }=\lambda_{\Theta_C}^ {-1} \circ (\pi ^ *)^ {\vee } \circ \lambda _{\Theta_D}.
\]
Then $(\pi^*)^ {\vee }\circ \lambda _{\Theta_D} \circ i$ is $0$ if and only if 
\[
 \lambda_{\Theta_C}^ {-1}\circ (\pi^ *)^ {\vee }\circ \lambda_{\Theta_D} \circ i=\Nm_{\pi} \circ \,i=0
\]
which is obvious.

Notice that the base locus of $\delta $ is
$$
Bs:=Bs(\delta)= \{ z \in P \ \mid \ \pi^* JC \subset \Theta_{D, z}\}.
$$
Our aim is to compute $Bs$ and to prove that under some genericity condition the covering $\pi$ can be recovered from $Bs$. First we notice that 
$\delta $ is in fact another representation of the map $ \varphi_{\Xi}$, hence $Bs$ is completely determined by the polarized Prym variety.

\begin{prop} \label{comm1}
There is a canonical isomorphism $i: \PP H^0(P, \cO_P(\Xi))^{\vee} \ra |2\Theta_C|$ making the following diagram commutative
$$
\xymatrix@C=2cm{
               & \PP H^0(P, \cO_P(\Xi))^{\vee} \ar[dd]^{\simeq}  \\
	      P \setminus Bs \ar[ur]^{\varphi_{\Xi}} \ar[dr]_{\delta}   & \\
	     & |2\Theta_C|      
    }
$$
In particular the base locus of $\varphi_{\Xi}$ coincides with the base locus of $\delta$. 
\end{prop}

\begin{proof}
The proof in  \cite[Proposition, p. 334]{mu} applies verbatim.  
\end{proof}

In order to compute explicitly the base locus $Bs$ it is convenient to work with a suitable translated of $P$ and of $\Xi$. We consider 
 $$
 P^{can}=  \Nm_{\pi}^{-1} (\omega_C \otimes \eta)  \subset \Pic^{2g-2+\frac r2}(D).
 $$
 Observe that $2g-2+\frac r2 =g(D)-1$.
 Then there is a canonical representative of the theta divisor
 $$
  \Xi^{can}= \Theta_D^{can} \cap P^{can}   
 $$
where $\Theta_D^{can} = W^0_{2g-2+\frac r2}(D) \subset \Pic^{g(D)-1}(D)$ is the canonical theta divisor.

Let $\kappa \in \Pic^{2g-2+\frac r2}(D)$ be such that $\Nm \kappa = \omega_C \otimes \eta$ and let $\Theta_D= \Theta^{can}_{D,-\kappa}$ be the translated theta divisor.  
Consider the translation $t_{-\kappa}: P^{can} \stackrel{\simeq}{\ra} P$, $L \mapsto L \otimes \kappa^{-1}$. 
Let $\widetilde{Bs}= t_{\kappa} (Bs)$ be the translated of the base locus in $P^{can}$, so we obtain the description:
\[
\widetilde{Bs}= t_{\kappa} (Bs)= \{ L \in P^{can} \ \mid \ \pi^*(JC) \subset \Theta^{can}_{D,-L} \}.
\]

From now on we suppose that $g\ge 2$. The case $g=1$ will be considered at the end of the proof.
We  denote by  $SU_C(2, \omega_C)$ the moduli space  of  S-equivalence classes of rank $2$
semistable vector bundles on $C$ with determinant $\omega_C$, which is a normal projective variety. 
Given a line bundle $L\in P^{can} \subset \Pic^{2g-2+\frac r2}(D)$, $\pi_*L$ is a rank $2$ vector bundle with 
$$
\det \pi_*L = \Nm (L ) \otimes \det \pi_*(\cO_D)= \omega_C \otimes \eta \otimes \det (\cO_C \oplus \eta^{-1}) = \omega_C. 
$$
 According to \cite[Lemma 1.2]{be_bsmf} there is an open set $P^{ss} \subset P^{can}$ such that $\pi_*L$ is semistable for all $L \in P^{ss}$, so
 the map $\pi_*:  P^{ss} \ra SU_C(2, \omega_C)$ is well defined. On the other hand,  there is a  {\it theta map} defined by
 $$
 \theta:   SU_C(2, \omega_C) \ra |2\Theta_C|, \qquad E \mapsto \theta(E),
 $$
 where $\theta(E)$ is a divisor whose support is $\{\alpha \in \Pic^0(C) \ \mid \  H^0(C, E \otimes \alpha) \neq 0\}$.  
 It is known that the theta map is well defined (\cite{r}), that is,  $\theta(E) $ is a divisor linearly equivalent to $2\Theta_C$,
 for every semistable rank 2 vector bundle with canonical determinant. 
 The following proposition is a consequence of several results in the literature (see for example \cite{vgi} and the references therein):
 
 \begin{prop}
 The theta map is an embedding $SU_C(2, \omega_C) \hookrightarrow |2\Theta_C|$ for a generic curve of genus $g\ge 2$.
 \end{prop}
 
In particular, every semistable rank 2-vector bundle with canonical determinant admits a theta divisor in $JC$ (\cite[Proposition 1.6.2]{r}).
Now we set $\tilde{\delta} =\delta \circ t_{-\kappa}$. With this notation we have:

 \begin{lem} \label{comm2}

 Then the  following diagram is commutative
 $$
\xymatrix@C=2cm@R=1.4cm{
                & SU_C(2, \omega_C) \ar@{^{(}->}[d]_{\theta} \\
	      P^{ss}   \ar[ur]^{\pi_*}  \ar[r]^(.5){\tilde{\delta}} \ar[d]_{t_{-\kappa}}    & |2\Theta_C| \ar[d]_{i}^{\simeq} \\
	      P \setminus Bs \ar[ur]^{\delta}  \ar[r]^(.4){\varphi_{\Xi}}& \PP H^0(P, \cO_P(\Xi))^{\vee} 	         
    }
$$

 \end{lem}
 
\begin{proof}
The commutativity of the lower triangle is shown in Proposition \ref{comm1}.
Let $L \in  P^{ss}  $, then $\pi_*L$ is a semistable rank 2 vector bundle with canonical determinant and the support of the divisor
$\theta(\pi_* L)$ is given by
\begin{eqnarray*}
\Supp \theta(\pi_*L) & = &\{ \alpha \in \Pic^0(C )\ \mid \ H^0(C, \alpha \otimes \pi_*L)  \neq 0 \} \\
& =  &  \{ \alpha \in \Pic^0(C )\ \mid \ H^0(D, \pi^* \alpha \otimes L)  \neq 0 \} \\
&  = & (\Theta^{can}_{-L})_{|{\pi^*JC}}\\
& = & \tilde{\delta} (L). 
\end{eqnarray*}
This proves the commutativity of the upper triangle. 
\end{proof} 

Observe that Lemma \ref{comm2} also shows that 
$$
\widetilde{Bs}  \subset B_{nss}:= \{ L \in P^{can}  \mid  \pi_*L  \textrm{ is not semistable} \}.
$$
On the other hand, if $\pi_*L$ is not semistable then there exists a  line subbundle $N \subset \pi_*L$ of slope $> g-1$. By Riemann-Roch 
we have $0 \neq H^0(C,N\otimes \alpha) \subseteq H^0(C, \pi_*L \otimes \alpha)$ for all $\alpha \in JC$ so  
$$
H^0(C, \pi_*L \otimes \alpha) = H^0(D, L \otimes \pi^ * \alpha ) \neq 0  \quad \forall  \alpha \in JC,
$$
that is, $\pi^*(JC) \subset \Theta_{D,-L}^{can}$ and $L \in \widetilde{Bs}$. This gives $ \widetilde{Bs}  =  B_{nss}$.
\bigskip

We consider now the subset of $P^{can}$:  
\[
B_0:= \{ L =\pi^*(A)(p_1+\dots +p_{\frac r2-2}) \ \mid \ A \in \Pic^g(C), \ p_i \in D,  \ \Nm L \cong \omega_C \otimes \eta   \}.
\]

\begin{prop} \label{base-locus}   The equality $B_0 = \widetilde{Bs}$ holds.
\end{prop}
\begin{proof}
Let $L= \pi^*(A')(\Sigma p_i)  \in B_0$, where $A'$ is effective of maximal degree $\ge g$. In particular, putting $\bar{p_i} = \pi(p_i)$ we have that $\bar{p_i}\ne \bar{p_j}$ for $i\ne j$.
In order to compute $H^0(D, \pi^ * \alpha \otimes L) $ with $\alpha \in JC$, we use a short exact sequence on $C$ given by Mumford in \cite{mu} 
(see the proof of the Proposition in page 338)  which in our case reads as follows:
\begin{equation} \label{exact_seq_Mumford}
 0\,\lra A' \lra \pi_*(L) \lra A'(\Sigma \bar {p_i})\otimes \eta^{-1} \lra \, 0,
\end{equation}
where $\pi_*(L) = \pi_*(\pi^*(A')(\Sigma p_i))$. 
Tensoring with $\alpha$ we get:  
\[
H^ 0(C, A'\otimes \alpha) \subset H^ 0(C, \pi_*L \otimes \alpha)=H^0(C,\pi_*(L\otimes \pi^*(\alpha)))=H^0(D,L\otimes \pi^*(\alpha)).
\]
Since the degree of $A'\otimes \alpha $ is at least $g=g(C)$, then $H^ 0(C, A'\otimes \alpha) \neq 0$. Therefore $\pi^* (JC) \subset \Theta^{can}_{D,-L} $ 
and this proves $B_0\subset \widetilde {Bs}$.
 
Let $L \in B_{nss}$. Then there exists a line bundle $A {\hookrightarrow} \pi_*L$  with 
$$
\deg A >  g-1 = \mu(\pi_*L):= \frac{ \deg \pi_*L }{ \rk \pi_*L }.
$$ 
Since  
\[
\Hom(A,\pi_*L)\cong \Hom(\pi^ *(A),L)
\]
there is a non-trivial map $\pi^ *(A) \to L$ which is necessarily injective.  We also obtain that $h^ 0(D, L\otimes \pi^ *(A^ {-1}))>0$ and hence 
$L$ is of the form $L=\pi^*A (\sum p_i)$, $p_i \in D$. Thus we have 
\[
 \widetilde {Bs}= B_{nss} \subset B_0
\]
and we are done.
\end{proof}

We use the description of the base locus to recover the covering $\pi:D \ra C$ from the Prym variety $(P, \Xi)$.
Define the variety
$$
W:=  \{ (A,\sum_i p_i) \in \Pic^g (C) \times D^{(\frac r2-2)}  \mid \ A^{\otimes 2} (\sum_i \bar{p_i}) \simeq \omega_C \otimes \eta \},
$$
where $\bar{p_i}:=\pi(p_i)$.

\begin{prop}
 Assume that $(C,\eta)$ is general and assume also that $r\ge \max \{6, \frac 23 (g+2)\}$. Then the natural morphism $W\lra B_0$ sending 
 \[
 (A,\sum_{i=1}^{\frac{r}{2} - 2} p_i) \mapsto \pi^*(A)(\sum_{i=1}^{\frac{r}{2} - 2}  p_i)
 \]
 is birational. 
\end{prop}

\begin{proof}
 The main difficulty of the proof is to show that for such  a generic element the equality $h^0(C,A)=1$ holds. To prove this  we denote by 
\[
 W_{\eta}=\{A\in \Pic^g(C) \mid h^0(C,\omega_C \otimes \eta \otimes A^{-\otimes 2}) > 0  \},
\]
which by definition is the image of the first projection $W\ra \Pic^g(C)$. Obviously if  
\[
\deg (\omega_C \otimes \eta \otimes A^{-\otimes 2})=2g-2+\frac r2 -2g =\frac r2 -2 \ge g, 
\]
then the condition on the cohomology is empty and $W_{\eta }=\Pic^g(C)$. In this case $h^0(C,A)=1$ generically for any $(C,\eta )$. This covers the 
cases $r\ge 2g+4$. Assume now that $r =2g+2$. Since the second projection $W\ra  D^{(\frac r2 -2)}$ is an \'etale covering, $\dim W_{\eta}\le \dim W =\frac r2 -2=g-1$,
 hence $W_{\eta}$ is a proper subvariety of $\Pic^g(C)$. We claim that it is not contained in the Brill-Noether locus $ W^1_g(C)$.  If we are able to prove that 
$\dim W_{\eta } =g-1$ then $W_{\eta } \not \subset W^1_g(C)\cong W^0_{g-2}(C)$. Since the cohomological condition that defines $W_{\eta }$ refers to $A^{\otimes 2}$, 
we consider $\overline {W}_{\eta}$ the image of the  \'etale map:
\[
 W_{\eta} \lra \overline {W}_{\eta} \subset \Pic^{2g}(C),\quad A \mapsto A ^{\otimes 2}.
\]
Then, an element $L \in \overline {W}_{\eta}$ satisfies 
\[
 0<h^0(C,\omega_C \otimes \eta \otimes L^{-1})=h^0(C,L\otimes \eta ^{-1})
\]
by Riemann-Roch. Then $\overline {W}_{\eta}$ is isomorphic to the translated $\eta + W^0_{g-1}(C)$ of the theta divisor on $\Pic^{g-1}(C)$. Thus for any 
$\eta \in T=\{\eta \mid \eta^{\otimes 2}\in W_{r}^0(C)\}$, $\dim W_{\eta}=g-1$, which proves the claim.

In the previous argument there is no restriction on $\eta $. In order to prove the statement for $r < 2g+2$ we shall show that for a generic $\eta $ in $T$
we have   $W_{\eta} \not \subset  W^1_g(C)$. We claim that the union of the subsets $W_{\eta}$, as $\eta $ varies in $T$, covers all $\Pic^g(C)$. 
In other words, if we define
\[ 
 I:=\{(A,\eta ) \in \Pic^g(C)\times T \mid A\in W_{\eta} \},
\]
this equivalent to show that the first projection $I \ra \Pic^g(C)$ is surjective. Let  $T=\{\eta \mid \eta^{\otimes 2}\in W_{r}^0(C)\}$ and
$T_0=\{ \eta \in T \mid |\eta^{\otimes 2}| \ \textnormal{ has a reduced divisor} \}$.
We fix an arbitrary element  $A\in \Pic^g(C)$. We ask for the existence of an $\eta \in T$ such that $h^0(C,\omega_C \otimes \eta \otimes 
A^{-\otimes 2})>0$. Then $\eta $ has to be chosen in the intersection of the Brill-Noether locus
\[
T_A:= W^0_{\frac r2-2}+(A^{\otimes 2}\otimes \omega_C^{-1})
\]
with $T$. If $r\ge g$, then $T=\Pic^{\frac r2}(C)$ and the existence of such an $\eta $ is obvious as far as 
$r\ge 6$. 
Assume from now on that $r\le g$, in particular $\dim T=\dim W^0_r(C)=r$. Since the cohomological class of $T_A$ is a fraction of a power of the theta divisor, 
the intersection $T_A\cap T$ is non empty if 
\[
 \dim T_A + \dim T -g= \frac r2 -2 +r-g\ge 0.
\]
Hence for $r\ge \max \{6, \frac 23 (g+2)\}$ we have the claimed surjectivity. Thus, for a generic $\eta \in T$, 
and therefore for a generic $\eta \in T_0$, the generic element  in $W_{\eta }$ satisfies $h^0(C,A)=1$.

\vskip 3mm
To finish the proof of the proposition we observe that the second projection $W\ra D^{(\frac r2-2)}$ is an \'etale covering of degree $2^{2g}$, hence for a 
generic element $ (A,\sum_i p_i)$ in $W$ we can simultaneously assume  that $\sum_i p_i$ is $\pi$-simple (i.e., it does not contain fibres of $\pi$) and $h^0(C,A)=1$. 
Let us consider again the exact sequence (\ref{exact_seq_Mumford})
\[
  0\,\lra A \lra \pi_*(\pi^*(A)(\sum_i p_i)) \lra A(\sum_i \bar {p_i})\otimes \eta^{-1} \lra \, 0.
\]
Since $A^{\otimes 2}(\sum_i \bar p_i)\cong \omega _C\otimes \eta $ the exact sequence becomes: 
\[
  0\,\lra A \lra \pi_*(\pi^*(A)(\sum_i p_i)) \lra \omega_C \otimes A^{-1} \lra \, 0.
\]
Notice that $H^0(C,\omega_C \otimes A^{-1}) \cong H^1(C,A)$. By Riemann-Roch Theorem this  group is trivial since $\deg(A)=g$ and $h^0(C,A)=1$. Therefore:
\[
 H^0(C,A) \cong H^0(C, \pi_*(\pi^*(A)(\sum_i p_i)) ).
\]
In other words, there is only one divisor in the linear series $| \pi^*(A)(\sum_i p_i) |$.
This immediately implies the generic injectivity of the map $W\ra B_0$. Since it is obviously surjective we are done.

\vskip 3mm
Now we consider the case $g=1$ and $r\ge 6$. We simply remark that the only place in the proof where we use  that $g\ge 2$ is in Lemma (\ref{base-locus}), 
where we use rank $2$ vector bundles on curves of genus at least $2$ to show the inclusion $\widetilde {Bs} \subset B_0$. Let us prove it directly:
 assume by  contradiction that $L \in \widetilde {Bs}$ does not belong to $B_0$, which in this case means that $L$ (which satisfies $h^0(C,L)>0$ by 
 definition) is  represented by an effective $\pi$-simple divisor. Then, the exact sequence (\ref{exact_seq_Mumford}) reads
\[
 0\lra \mathcal O_C \lra \pi_*(L) \lra \mathcal O_C \lra 0,
\]
hence for any $\alpha \in JC\setminus \{0\}$ we obtain 
\[
0=h^0(C,\pi_*(L)\otimes \alpha )=h^0(C,\pi_*(L\otimes \pi^*\alpha) )=h^0(D,L\otimes \pi^*(\alpha)), 
\] 
which is a contradiction since $L \in  \widetilde {Bs}$.
\end{proof}

\vskip 3mm
Now we are in disposition of finishing the proof of the main Theorem of this section. According to the last Proposition the quotient of the action of 
$JC_2=\Ker (\lambda _{\Xi})$ on the base locus $Bs$ of the linear system $|\Xi |$ is birational to $W/JC_2 = D^{(\frac r2 -2)}$. Since 
$g(D)-1 =2g-2+\frac r2 >\frac r2 -2$ we can apply the main Theorem in \cite{mar} which gives a constructive method to recover the curve 
$D$ from a variety birational  to $D^{(\frac r2 -2)}$. Moreover the base locus $Bs$ is symmetric in $P$ and this symmetry corresponds to 
the involution $\sigma ^{(\frac r2-2)}$ on the symmetric  product of the curve. Applying again the result on \cite{mar} we recover also the involution 
$\sigma $ on $D$ and therefore the whole covering $D\ra D/\langle \sigma \rangle$.

\begin{rem} \label{key-obs}
The case of  Corollary \ref{corolario-Torelli} ($r=6$ and $g=2$) is particularly simple. It is not hard to see that in this case, for a generic $\eta $, 
$W$  and $B_0\cong Bs$ are  isomorphic irreducible curves of genus $81$. In this case the condition on $\eta $ can be precised: we need 
the vanishing $h^0(C, \eta \otimes \omega_C^{-1})=0$. 
The quotient of the action of $\Ker (\lambda_{\Xi})$ on $W$  gives directly the curve $D$ and the symmetry on $W$ induces the involution 
$\sigma $ on $D$. Finally $D/\langle \sigma \rangle =C$.
\end{rem}

\section{Case $g=5$ and $r=2$}

In this section we take care of the map $\cP_{5,2}:\cR_{5,2} \ra \cA_5$ which by Corollary (2.3) in \cite{mp} 
is generically finite. This is the last 
open case of the generic Torelli problem stated in the introduction. We will prove at the end of this section that 
$\mathcal P_{5,2}$ is generically injective.

In order to study the degree of $\cP_{5,2}$ we use the classical Prym map $\cP_6:\mathcal R_6 \ra \mathcal A_5$ defined on irreducible unramified 
double coverings of curves of genus $6$. To do so we first recall how this map was extended by Beauville in \cite{be_invent} to a proper map:
\[
 \overline {\cP_6}:\overline {\mathcal R}_6 \lra \mathcal A_5. 
\]
We give the basic definitions for any $g$. To start with, we borrow from \cite{be_invent} the description of the elements belonging to 
$\overline {\mathcal R}_g$, called admissible coverings (see condition (**) and Lemma (5.1) in loc. cit.). 
\begin{defn} Let $\tilde C$ be a connected curve with only ordinary double points and arithmetic genus $p_a(\tilde C)=2g-1$, and let $\sigma $ be an involution 
on $\tilde C$. Then $\tilde C \lra \tilde C/\langle \sigma \rangle$ is an admissible covering if and only if the following conditions are fulfilled:
 \begin{enumerate}
  \item There are not fixed non-singular points of $\sigma $.
  \item At the nodes fixed by $\sigma $ the two branches are not exchanged.
  \item The number of nodes exchanged under $\sigma $ equals the number of irreducible components of $\tilde C$ exchanged under $\sigma$.
 \end{enumerate}
\end{defn}
In particular $\sigma$ is not the identity on any component.
Under these conditions the arithmetical genus of $\tilde C/\langle \sigma \rangle$ is $g$ and the Prym variety attached to the covering can be defined in a similar 
way of the standard Prym construction and it is a ppav.

An instance of admissible covering is the following: consider two copies of a smooth curve $C$ of genus $g-1$ each one with two marked points. 
Call $p_1, q_1$ the marked points in the first copy and $p_2, q_2$ the same points in the second copy. Then  the curve
\[
 \tilde C = C \cup C / (p_1\sim q_2,p_2\sim q_1)
\]
has a natural involution $\sigma$ exchanging the components. Observe that $\tilde C/\langle \sigma \rangle$ is the irreducible nodal curve 
$C/ p_1 \sim p_2$. This is an admissible covering whose associated Prym variety is the Jacobian of $C$. These elements
$(\tilde C, \sigma )$ are called Wirtinger coverings. Observe that the closure of the set parametrizing these objects describe a divisor in 
$\overline {\mathcal R}_g$ which is birational to the moduli space $\mathcal M_{g-1,2}$ of curves of genus $g-1$ with two marked points. 
Hence this divisor is irreducible. We denote it by $\Delta^{W}$.

Now we come back to our particular situation, hence $g=6$.
Let $\pi:D\ra C$ be a covering  in $\mathcal R_{5,2}$. By glueing in $C$ the two branch points and in $D$ the two ramification points we get an 
admissible covering. These curves form a  divisor in the boundary of $\overline {\mathcal R}_{6}$, birational to $\mathcal R_{5,2}$ (hence irreducible) 
that we denote by $\Delta^n$. A generic element in this divisor is an irreducible admissible covering of a curve with exactly one node. 
Moreover, the composition 
\[
 \mathcal R_{5,2} \hookrightarrow \Delta ^n \hookrightarrow \overline {\mathcal R}_{6}  {\overset{\overline {\cP_6}} \longrightarrow } \mathcal A_5
\]
is the Prym map $\cP_{5,2}$ we are considering.

\begin{rem}
Let us consider the forgetful map $\pi :\mathcal R_6 \ra \cM_6$ sending $[\tilde C\ra C]$ to the class of the curve $C$. Let $\overline {\cM }_6$ the 
Deligne-Mumford compactification of the moduli space of curves. Denote by $\overline {\cR}_6'$ the normalization of  $\overline {\cM}_6$ in the 
function field of $\cR_6$. Then there is a map $\pi': \overline{\cR}_6'\ra \overline {\cM}_6$ which has been studied in \cite{fl}. The preimage of the 
divisor of nodal curves $\Delta_0 \subset \overline {\cM}_6$ under $\pi'$ is 
\[
 \pi'^*(\Delta _0)=\Delta^{na} + \Delta ^W +2 \Delta ^n,
\]
where  the general element $\tilde C \ra C$ of $\Delta^{na}$ is not admissible (in  Beauville's sense) since $\tilde C$ is irreducible with two nodes
interchanged by $\sigma $.  From this one sees that $\pi'$ is ramified along $\Delta^n$ (which explains the notation $\Delta ^{ram}$ in \cite{fl}).
Of course, there are other components of codimension $1$ in the boundary  of $\overline {\cR}_6'$ (and of $\overline{\cR}_6$) that are in the 
preimage of $\Delta_{i,6-i}\subset \overline{\cM}_6$.  We refer to \cite{fl} and the references thererin for more details on this map for any $g$.
\end{rem}

In order to study the degree of $\cP_{5,2}$ our strategy is to look at the preimage of the locus $\mathcal C$, the set of the smooth cubic
 threefolds embedded in $\cA_5$ via the Intermediate Jacobian map (see \cite{cg}). We start by recalling some basic facts on the geometry of cubic 
 threefolds and the representation of its Intermediate Jacobian as Prym variety. All the results are very classical and can be found in \cite{cg}, \cite{ds}, 
 \cite{do} and \cite{murre}.

Let $V$ be a smooth cubic threefold in $\mathbb P^4$. The Intermediate Jacobian $JV$  of $V$  is isomorphic to the Prym variety of many 
elements in $\overline {\cR}_6$. 
This is done as follows: given a line $l\subset V$ and a $2$-plane $\Pi \subset \mathbb P^4$ disjoint with $l$, the projection $V\setminus {l}\ra \Pi$ 
extends to a morphism on the blowing-up $ {V_l}$ of $V$ on $l$. The fibres of this map are conics and the discriminant curve on $\Pi$ parametrizing
 the points where these 
conics degenerate is a quintic $Q_l$. The pairs of lines on the degenerate fibres give a curve $\tilde Q_l$ in the Grassmannian of lines in $\mathbb P^4$.
 Then $\tilde Q_l \ra Q_l$ is an admissible covering and its Prym variety is isomorphic as ppav to $JV$. Hence  $\overline {\cP_6}^{-1}(V)$ contains the 
 set of all these coverings of plane quintics, which is in bijection with the Fano surface $F(V)$ of the lines contained in $V$. In \cite[Part V, \S 1]{ds} it is 
 proved that in fact the fibre is exactly $F(V)$ for any $V$. Since we are interested in the fibre restricted to the divisor $\Delta^n$ we need to know in 
which cases the quintic $Q_l$ is not smooth, in other words, the first step is to compute
\[
 F(V) \cap \Delta^n
\]
for any cubic threefold $V$.

This is related to the geometry of the cubic threefolds due to the following result of Beauville:
\begin{prop}(\cite[Proposition 1.2]{be_pryms_et_ji})\label{beauville_conic_bundles}
 Let $V_l \ra \Pi$ be a conic bundle as above and let $Q_l$ be the discriminant curve. Then
 \begin{enumerate}
  \item The curve $Q_l$ has at most ordinary double points as singularities. 
  \item If $s$ is a regular point in $Q_l$, then the corresponding conic has only one singular point, i.e. is formed by two different lines.
  \item If $s$ is a node of $Q_l$, then the corresponding conic is a double line.
 \end{enumerate}
\end{prop}
Hence $Q_l$ is singular if and only if there is a plane in $\mathbb P^4$ intersecting $V$ in $l+2r$, where $r\in F(V)$. So we are interesting in the following two sets of lines:
\[
\begin{aligned}
 &\Gamma =\{l\in F(V) \mid \exists \text{ a plane } L \text { and a line } r\in F(V) \text{ with } V\cdot L=l+2r\}, \\
 &\Gamma '=\{r\in F(V) \mid \exists \text{ a plane } L \text { and a line } l\in F(V) \text{ with } V\cdot L=l+2r\}.
 \end{aligned}
 \]
According to Proposition \ref{beauville_conic_bundles}, the set $\Gamma$ parametrizes the lines $l$ for which $Q_l$ is singular.

The curve $\Gamma'$ has deserved more attention in the literature. It appears in \cite[section 6]{cg} as the set of lines of ``second type''. 
In Proposition 10.21 of that paper it is shown that $\Gamma '$ has pure dimension $1$ and, as divisor on the Fano surface, it is linearly equivalent to 
twice the canonical divisor. Moreover, Murre proved that $\Gamma '$ ($\mathcal F_0$ in Murre's notation) is smooth, non-necessarily connected 
and $\Gamma$ ($\mathcal F_0'$ with his notation) is Zariski-closed of dimension at most $1$ (see \cite[Corollary (1.9), Lemma (1.11)]{murre}). 

It is easy to go further with the techniques of \cite{murre} to show that in fact $\Gamma $ is also a curve. 
To this end, let us consider the incidence variety 
\[
I=\{(l,r) \in \Gamma \times \Gamma' \mid \text {there is a plane } L \text{ with }  L\cdot V=l+2r\}.
\]
By Proposition \ref{beauville_conic_bundles}  the discriminant curve $Q_l$ has at most a finite number of singularities, hence we have that 
$I\ra \Gamma $ is finite-to-one. The following lemma implies that  $I\ra \Gamma '$ is bijective.
\begin{lemma}
Given a line $r\in \Gamma '$ there is only one line $l$ such that there is a plane $L$ with $L\cdot V=l+2r$.
\end{lemma}
\begin{proof}
We fix $r$, $l$ and $L$ as in the statement. We assume $r\neq l$, a similar proof works with $r=l$.
We can choose a reference system in such a way that $(1:0:0:0:0), (0:1:0:0:0)\in r$ and $(1:0:0:0:0), (0:0:1:0:0)\in l$. Hence, using coordinates 
$x,y,z,u,v$, the  equations of $r$ and  $l$ are  $z=u=v=0$ and $y=u=v=0$ respectively. Moreover $L$ has equations $u=v=0$. 
Our hypothesis implies that the cubic polynomial $F$ defining $V$ is of the form:
\[
F = c\, y z^2 \, + u \, G_2 \, + \, v \,H_2,
\]
where $G_2,H_2$ are the equations of two quadrics $Q_1,Q_2$ in $\mathbb P^4$. Observe that if $c=0$ then, by an easy computation, the points satisfying
$G_2=H_2=u=v=0$ are singularities of $V$, which is a contradiction. So we can assume $c=1$.

By convenience, in the later computation we decompose the quadratic forms $G_2,H_2$ separating the part in the variables $x,y$:
\[
\begin{aligned}
&G_2\,=\, A(x,y)\,+\,z\,L_1+\,u \,L_2 + \,v\, L_3 \\
&H_2\,=\, B(x,y)\,+\,z\,M_1+\,u \,M_2 + \,v\, M_3,
\end{aligned} 
\]
where $L_i,M_i$ are linear forms.
Hence we have
\begin{equation}\label{cubic_equation}
\begin{aligned}
F\,=\, &  y z^2 \, + u \, (A(x,y)\,+\,z\,L_1+\,u \,L_2 + \,v\, L_3)     \, + \, \\ & v \,(   B(x,y)\,+\,z\,M_1+\,u \,M_2 + \,v\, M_3).
\end{aligned}  
\end{equation}
  
The planes through $r$ are parametrized by the points of the plane $x=y=0$. We want to know for which points $p=(0:0:\alpha:\beta:\gamma)$ 
the plane $r\vee p$ intersects $V$ in $2r+l'$ for some line $l'$.  
A generic point of $r\vee p$ is of the form $(r:s:t \alpha : t \beta : t\gamma )$ for some parameters $r,s,t$. Replacing in (\ref{cubic_equation}) 
we have to impose that $t$ (the equation of $r$ in the plane) appears with multiplicity $2$:
\[
\alpha ^2st^2+ t \beta (A(r,s)+t(\alpha L_1+\dots))
+            t\gamma (B(r,s)+t(\alpha M_1+\dots )).         
\]
So the condition is $\beta A(r,s)+ \gamma B(r,s)=0$ for any $r,s$. The next claim implies that $\beta=\gamma =0$ and therefore the only solution 
corresponds to the plane $L$ and hence $l$ is unique.

{\bf Claim:} If the quadratic forms $A(x,y), B(x,y)$ are not linearly independent, then $V$ is singular.

Indeed, assume that $B(x,y)=\lambda A(x,y)$ for some constant $\lambda$. Then $F$ is of the form:
\[
y z^2 \, + u \, (A(x,y)\,+\,z\,L_1+\,u \,L_2 + \,v\, L_3)     \, + \, v \,(  \lambda A(x,y)\,+\,z\,M_1+\,u \,M_2 + \,v\, M_3).
\]

One easily checks that the points which satisfy $A(x,y)=z=u=v=0$ are non-trivial solutions of the system of the partial derivatives of $F$, hence $V$ 
is singular.
\end{proof}

\begin{corollary}
 For any smooth cubic threefold, $\Gamma $ is a curve dominated by $\Gamma'$.
\end{corollary}

In the case of general cubic threefold we can go further in the description of the properties of the curves $\Gamma $ and $\Gamma '$ 

\begin{prop}
 For a general cubic threefold $V$ the curves $\Gamma $ and $\Gamma '$ are irreducible  and $\Gamma '$ is the normalization of $\Gamma$.
\end{prop}
\begin{proof}
We first prove the irreducibility. This is a consequence of the fact that for a general $V$ the Fano surface $F(V)$ has Picard number $1$ (see \cite[p. 382]{rou}).
Indeed, recall from \cite[section 10]{cg} that $\Gamma '$ is bicanonical in $F(V)$ and that canonical divisor is very ample giving the Pl\"ucker 
embedding $F(V)\hookrightarrow Grass(1,4)\hookrightarrow \mathbb P^9$. Therefore the positive generator of the Neron-Severi group is ample. 
Thus, if $\Gamma'= T_1 \cup T_2$ then $T_1 \cdot T_2>0$ contradicting the smoothness proved by Murre. Since $\Gamma '$ dominates $\Gamma$
both are irreducible.  We only have to prove that the map $\Gamma ' \lra \Gamma$ has degree $1$. 
Following (5.13) in \cite{do}, there is an involution $\lambda $ in the fibre of a generic $JV\in \mathcal C$ in such a way that, given a line $l\in \Gamma$, the preimages $r\in \Gamma'$ corresponds to odd semicanonical pencils in a smooth Prym curve $(C,\eta)=\lambda(\tilde Q_l,Q_l)$ (that is, theta characteristics $L$ with $h^0(C,L)=2$ and $h^0(C,L\otimes \eta)=1$). The locus of curves with a semicanonical pencil is an irreducible divisor and a generic element of this divisor has only one semicanonical pencil, this implies our result. 
%
\end{proof}

\vskip 3mm
We come back to our computation of the degree of the map $\cP_{5,2}$ identified with $\overline {\cP_6}_{|\Delta ^n}$. Observe that, by the previous
results, any $JV\in \mathcal C$ can be represented as $P(\tilde Q_l,Q_l)$ for a quintic plane curve with only one node. Notice that $\tilde Q_l$ must be
irreducible, otherwise $\tilde Q_l \ra Q_l$ would be a Wirtinger cover and $JV$ would be a Jacobian contradicting the main result in \cite{cg}.
 Summarizing, we have that:
\begin {enumerate}
 \item [(a)] The locus $\mathcal C$ is contained in  $\overline {\cP_6}(\Delta ^n) $.
 \item [(b)] The fiber over $V$ of $\overline {\cP_6}_{|\Delta ^n}$ is $\Gamma = \Delta ^n \cap F(V)$.
 \item [(c)] The locus $\mathcal S' = (\overline {\cP_6}_{|\Delta ^n})^{-1}(\mathcal C)$ has dimension $11$, all the fibers of $\mathcal S' \ra \mathcal C$ 
 are $1$ dimensional and the generic fibre is irreducible.  
\end {enumerate}

In particular the irreducible component $\mathcal S\subset \mathcal S'$ containing the generic irreducible fibers is the only component of 
$\mathcal S'$ of dimension $11$. Since the map is closed $S=S'$.

Although the map  $\overline {\cP_6}_{|\Delta ^n}$ is generically finite, the fibres over $\mathcal C$ are one dimensional as we have already seen. 
Nevertheless, the degree of $\overline {\cP_6}_{|\Delta ^n}$ can be computed at this locus by using the concept of local degree explored in 
\cite[Part I, section 3]{ds}. The local degree $d$ of $\overline {\cP_6}_{|\Delta ^n}$ along $\mathcal S$ is the degree of map obtained from 
$\overline {\cP_6}_{|\Delta ^n}$ by localizing at $\mathcal S$ in the source and at $\mathcal C$ in the target:
\[
 (\Delta ^n)_\mathcal S \lra (\mathcal A_5)_\mathcal C.
\]
We have that $\deg(\overline {\cP_6}_{|\Delta ^n})=d$. 

Assume that $\overline {\cP_6}_{|\Delta ^n}$ lifts to a regular map from the blowup $\widetilde \Delta^n$ of $\Delta^n$ along $\mathcal S$ to the blowup 
of $\widetilde \cA_5$ of $\cA_5$ along $\mathcal C$, sending the exceptional divisor $\tilde {\mathcal S}$ to the exceptional divisor $\tilde {\mathcal C}$:
\[
 \widetilde {\overline {\cP_6}}_{|\Delta ^n}:\tilde {\mathcal S} \lra \tilde {\mathcal C}.
\]
Then:
\begin{lemma}\cite[Lemma 3.2]{ds} \label{local_degree}
 If the natural linear map  $N_{\mathcal S|\Delta^n,s}\lra N_{\mathcal C|\cA_5,\overline {\cP_6}(s)}$ induced by 
 \[
 d(\overline {\cP_6}_{|\Delta ^n}):T_{\Delta^n}\lra T_{\cA_5}
  \]
is injective at any $s\in \mathcal S$, then the local degree at $\mathcal S$ equals the degree of $\widetilde {\overline {\cP_6}}_{|\Delta ^n}:\tilde 
{\mathcal S} \lra \tilde {\mathcal C}$. 
\end{lemma}
In other words, under some conditions on the behavior of the normal bundle, we can compute the degree by looking at the degree of the map between
 exceptional divisors. In part V of \cite{ds} the authors proved that the conditions of the lemma are satisfied by the usual Prym map $\overline {\cP_6}: 
 \overline {\cR}_6\ra \cA_5$ with respect to $\mathcal C$. Moreover they give a nice geometrical description of the blowups involved in the picture. 
 In this way they are able to confirm once again that the degree of $P$ is $27$ (see also \cite[\S 4]{do}). The next proposition is a summary of the 
 results in \S1 and \S2 in part V of \cite{ds}: 
 
\begin{proposition} Let $\cF$ be the closure of the union of all the Fano surfaces $F(V)$ of smooth cubic threefolds $V$. Denote by $\pi_1:\widetilde 
{\overline R_6}\ra \overline R_6$ the blowup of $\overline R_6$ along $\cF$ and by $\pi_2: \widetilde {\cA_5}\ra \cA_5$ the blowup of $\cA_5$ along $\mathcal C$. Then:
 \begin{enumerate}
 \item $\overline {\cP_6}^{-1}(\mathcal C)=\cF$.
 \item For a cubic threefold $V$ the fibre $\pi_2^{-1}(V)$ can be identified with the dual of the ambient space $\mathbb P^4$ of $V$.
 \item For an admissible covering $(\widetilde {Q_l}, Q_l)\in \cF$ the fibre $\pi_1^{-1}(\widetilde {Q_l},Q_l)$ can be identified with the dual of
 the ambient space $\mathbb P^2$ of $Q_l$.
 \item The map between exceptional divisors $\tilde {\cF} \ra \tilde {\mathcal C}$ sends $(\tilde Q_l,Q_l,m)$ to $(V,l\vee m)$. In particular is 
 well-defined and injective everywhere and the conditions of the Lemma (\ref{local_degree}) are satisfied.
\end{enumerate}
\end{proposition}
\begin{corollary} Let $\mathcal S$ be the fibre of $\mathcal C$
by the map $\overline {\cP_6}_{|\Delta ^n}:\Delta^n\ra \cA_5$, and let $\tilde {\mathcal S}$ the exceptional divisor of the blowup of $\Delta^n$ along 
$\mathcal S$. Then the degree of $\overline {\cP_6}_{|\Delta ^n}$ equals the degree of the map $\tilde  {\mathcal S} \ra \tilde {\mathcal C}$ sending 
$(\tilde Q_l,Q_l,m)$ to $(V,l\vee m)$ 
\end{corollary}
\begin{proof}
Since $\Delta ^n\subset \overline {\cR}_6$ is a closed immersion, the universal property of blowups (see \cite[Corollary 7.15]{har}) implies that the 
previous Proposition is valid for the restriction to $\Delta^n$. In particular,
the conditions of the Lemma \ref{local_degree} are satisfied and the local degree (hence the degree) can be computed looking at the map between 
exceptional divisors.  \end{proof}
Now we can prove the main result of this section: 
\begin{thm} \label{gen_injective}
 The Prym map $\cP_{5,2}: \mathcal R_{5,2}\ra \cA_5$ is generically injective. 
\end{thm}
\begin{proof}
 Due to the last corollary, the theorem reduces to check the following fact: Let $V$ be a generic smooth cubic threefold and let $l\subset V$ 
 a generic line 
 on $\Gamma $ (hence $Q_l$ is a quintic nodal curve). Given a a generic hyperplane $H$ in $\mathbb P^4$ containing $l$, there is no other line in 
 $\Gamma $ contained in $H$. 
 
 To prove this we consider 
 \[
\mathcal H_0=\{(l,H)\in \Gamma \times \mathbb P^{4*} \mid l\subset H \}.  
 \]
 This a $3$-dimensional closed subvariety, since all the fibres of $\mathcal H_0 \ra \Gamma $ are isomorphic to $\mathbb P^2$. The variety we are 
 interested in 
is $\mathcal H=p_2(\mathcal H_0)$. The image of the rational map
 \[
 \Gamma \times \Gamma \dasharrow \mathcal H                                                                                                                                                                                                                         
 \]
 sending $(l_1,l_2)$ to the hyperplane $l_1\vee l_2$ is a $2$-dimensional family $\mathcal H_2$  of hyperplanes contained in $\mathcal H$. To finish the proof 
of the theorem it is enough to show that $\dim \mathcal H=3$. Indeed, otherwise for a generic $H\in \mathcal H$ there are infinitely many lines of $\Gamma $  
contained in $H$, thus the whole  $\Gamma $ give lines in $H$. Taking three different linearly independent hyperplanes $H_1, H_2, H_3\in \mathcal H$ 
such that  all the lines  of $\Gamma $ are in the three hyperplanes we get a contradiction.
 \end{proof}

\vspace{0.7cm}
 
\appendix
\section{\\Geometry of $\cR_{2,6}$ via Kummer quartic surfaces \\ by Alessandro Verra}

In this appendix we describe the moduli space $\mathcal R_{2,6}$ via the geometry of the sextic models in $\mathbf P^4$ of a smooth, integral genus
two curve $C$ and of its naturally associated Kummer quartic surface.  Our goal is to derive from this some recreational geometry and to prove the 
following result.
 \begin{theorem} \label{rationality} 
 $\mathcal R_{2,6}$ is rational. 
 \end{theorem}
To perform the proof let us introduce some preliminaries. Recall that a point of $\mathcal R_{2,6}$ is defined by a triple $(C, \eta, b)$ where $C$ is a curve
as above, $\eta \in \Pic^3C$ and $b$ is a smooth element of  $\vert \eta^{\otimes 2} \vert$.   Since $\eta^{\otimes 2}$ has degree $6$ we have $\dim \vert
 \eta^{\otimes 2} \vert = 4$. Furthermore, a well known theorem of Mumford (\cite{mu1}) implies that $\eta^{\otimes 2}$ defines a projectively normal 
 embedding 
$$
C \subset \mathbf P^4,
$$
with ideal sheaf $\mathcal I_C$  generated by quadrics. The linear system $\vert \mathcal I_C(2) \vert$ is $3$-dimensional and reflects the properties of
the Kummer surface of $\Pic^0 C$. We recollect from \cite{br} some of these properties. First observe that  $C$ is contained in the scroll
$$
R = \cup \langle x + \sigma(x) \rangle \ , \ x \in C,
$$
where $\sigma$ is the involution on the covering exchanging the sheets over $C$.  So $R$ is either a cone over a rational normal cubic or a 
smooth cubic scroll,  biregular to $\mathbf P^2$ blown up at one point.  Notice that $R$ is a cone if and only if $C$ is 3-canonically embedded. 
In particular the moduli points of triples $(C, \eta, b)$ such that $R$ is not a cone fill up a dense open set of $\mathcal R_{2,6}$. Therefore,  it will
be not restrictive to assume that $R$ is smooth. \par   
The ideal sheaf $\mathcal I_R$ of $R$ is generated by quadrics and $\dim \vert \mathcal I_R(2) 
\vert = 2$. This implies that $C = R \cdot Q$, where 
$Q$ is a quadric smooth along $C$. Hence, by Bertini theorem, it follows that a general $Q \in \vert \mathcal I_C(2) \vert$ is smooth. Then the discriminant 
locus of $ \vert \mathcal I_C(2) \vert$ is a surface, namely the quintic surface
$$
D := \lbrace Q \in \vert \mathcal I_C(2) \vert \ \vert \ \Sing Q \neq \emptyset \rbrace.
$$
Notice that $D$ is not integral. Indeed, by Lefschetz hyperplane theorem, no quadric  containing $R$ is smooth. Hence the net $\vert \mathcal I_R(2) \vert$ 
is a plane contained in $D$. From now on we will use the following notation:
$$ 
\mathbf P^3 := \vert \mathcal I_C(2) \vert \quad \text {and} \quad \mathbf P^2 := \vert \mathcal I_R(2) \vert. 
$$  
We describe $\mathbf P^2$ explicitly. The surface $R$ contains a unique exceptional line that we will denote by $E$. We can associate to each point 
$v \in R$  a unique quadric $Q_v$ containing $R$ in the following way.  Let $p_v: R \to \mathbf P^3$ be the projection from $v$. Then $p_v(R)$ is a 
quadric surface and the cone of vertex $v$ over $p_v(R)$ is a quadric  $Q_v$ through $R$.  If  $Q'_v$  is  a second quadric through $R$ singular at $v$,
then $Q_v \cap Q'_v$ is a cone of vertex $v$ over a curve, which implies that  $R$ is 
a cone, so we get a contradiction. Therefore the assignment $v \to Q_v$ defines a morphism 
$$ 
\beta: R \to \mathbf P^2. 
$$
\begin{proposition} \label{contraction}
The map $\beta$ is the  contraction of $E$. Moreover either $v \not \in E$ and $\beta(v)$ has rank $4$ or $\beta(E)$ has rank $3$. In the latter 
case $E = \Sing \beta(E)$.    
\end{proposition}

\begin{proof} Since the arguments are standard we only sketch the proof.  Let $v \in R$ then $p_v(R)$ is obtained via the elementary transformation of 
 center $v$ of the ruled surface $R$ ( \cite{hh}). Hence $p_v(R)$ is a smooth quadric surface if $v \not \in E$. If $v \in E$ then $p_v(R)$ has rank $3$, 
 $p_v(R)$ is singular at $p_v(E)$ and $\Sing Q_v = E$. Now it is easy to see that there is a unique rank $3$ quadric $Q_E$ containing $R$. Hence we
  have  $Q_v = Q_E$ and $\beta$ is constant on $E$. Then the statement follows.
 \end{proof}
 To continue the study of $D$ we consider  the natural involution
 $$ 
 s: \Pic^3 C \to \Pic^3C 
 $$ 
defined by  $s(L) := \eta^{\otimes 2}(-L)$,  so  $\Pic^3 C / \langle s \rangle$ is
biregular to the Kummer surface of $\Pic^0 C $.  Let $L \in \Pic^3 C$ then $\dim \vert L \vert = 1$ and each $d \in \vert L \vert$ satisfies $\dim \langle d 
\rangle = 2$. Otherwise  we would have $ \dim \vert \eta^{\otimes 2}(-d) \vert \geq 2$, which is impossible. This defines a quadric through $C$, namely
$$
Q_L := \bigcup_{d \in \vert L \vert} \langle d \rangle,
$$
the union of the planes generated by the divisors in $|L|$.  Let 
$$
u: \Pic^3 C  \to \mathbf P^3, \qquad  L \mapsto Q_L
$$
and $S := u(\Pic^3 C)$. Since $Q_L$ is ruled by planes, then $Q_L$ is singular and hence $u(\Pic^3 C) \subset D$.
 \begin{lemma} 
 At a general $Q_L$ the fibre of $u: \Pic^3 C \to S$ is $\lbrace L, s(L) \rbrace$. 
 \end{lemma}
\begin{proof}  From $L \otimes s(L) \cong \mathcal O_C(1)$ it easily follows that $Q_L = Q_{s(L)}$. Let $\vert L \vert$ be a base-point-free pencil, 
then  $L$ is uniquely  reconstructed from the ruling of planes $\langle d \rangle$ of $Q_L$. Since for a general $L \in \Pic^3 C$, $\vert L \vert$ and 
$\vert s(L) \vert$ are base point-free distinct pencils, the statement follows. 
\end{proof}
To complete the description of $S$ we consider the natural theta divisor 
$$ 
\Theta := \lbrace M \in \Pic^3 C \ \vert \ M \cong \omega_C(x), \ x \in C \rbrace 
$$
and its conjugate by the involution $s$
$$
\Theta_{\eta} := s(\Theta) =  \lbrace \eta^{\otimes 2}(-M) \ , \ M \in \Theta \rbrace.
$$
Observe that $M \in \Theta$ if and only if $\vert M \vert$ is base-point-free.
\begin{theorem}  \ 
\begin{enumerate}
\item It holds $u^* \mathbf P^2 = \Theta + \Theta_{\eta}$;
\item the linear system $\vert \Theta + \Theta_{\eta} \vert$ defines the map $u: \Pic^3C \to \mathbf P^3$;
\item the surface $S$ is the Kummer quartic model of $\Pic^3 C / \langle s \rangle$.
\end{enumerate}
\end{theorem}
\begin{proof}  (1) Let $Q \in \mathbf P^2$ be of rank 4 then $\Sing Q$ is a point $o \in R$. Consider the projection $p_o: R \to \mathbf P^3$ from $o$, then
 $p_o(R)$ is a smooth quadric and one of its rulings is the image of the ruling of lines of $R$.  If $o \not \in C$ then $p_o(C)$ is a curve of type (2,4) in 
 $p_o(R)$ and the two rulings of planes of $Q$ cut $\vert \omega_C \vert$
and a base point free pencil of degree 4 on $C$. Hence $Q$ is not in $S$.  If $o \in C$ then $Q = u(L)$, where $L $ is
$\omega_C(o)$ and then $L \in \Theta$ and $s(L) \in \Theta_{\eta}$.  Secondly,  let $Q$ be of rank 3 then $\Sing Q \cap C = \lbrace o_1, o_2 \rbrace$ 
and one has $Q = u(L)$,  where 
$ \lbrace L, s(L) \rbrace = \lbrace \omega_C(o_1), \omega_C(o_2) \rbrace \subset \Theta$. The discussion implies  $ u^{-1}(\mathbf P^2) = \Theta \cup 
\Theta_{\eta}$.
Now $u: \Pic^3C \to S$ is a morphism of degree 2 and $\deg S \leq 4$, since $D$ has degree 5. This easily implies $\deg S = 4$ and $u^* \mathbf P^2 = 
\Theta + \Theta_{\eta}$.  
This shows (1) and (1) implies (2) and (3). \end{proof}

It is well known that $\Sing S$ consists of sixteen nodes. In our construction these are the images of the fixed points of $s$, in other words
$$
\Sing S = \lbrace Q_L \ \vert \ L^{\otimes 2} \cong \mathcal O_C(1) \rbrace.
$$
Since $L = s(L)$ the quadric $Q_L$ has a unique ruling of planes and rank $3$. Hence,  $Q_{\eta} \in \Sing S$ since  $s(\eta) = \eta$, and the rank $3$
quadric $Q_{\eta}$ defines a node of $S$.

For a general triple $(C, \eta, b) \in \mathcal R_{2,6}$ is not restrictive to assume $\eta \not \in \Theta$. Then $\vert \eta \vert$ is base point free and 
$\Sing Q_{\eta} \cap C = \emptyset$, where $\Sing Q_{\eta}$ is a line. By Proposition \ref{contraction} this implies that  
$$
C = Q_{\eta} \cdot R.
$$
Since $R$ is unique in $\mathbf P^4$ up to projective automorphisms, we fix it. Then a general pair $(C, \eta)$ is defined
via the above construction by a rank $3$ quadric not containing $R$.  \par
\begin{proof}[Proof of Theorem~\ref{rationality}]
Let $G(1,4)$ be the Grassmannian of lines of  $\mathbf P^4$; we use the same notation for a line $\ell \subset \mathbf P^4$ and its parameter 
in $G(1,4)$. Let $\mathcal I_{\ell}$ be the ideal sheaf of $\ell$. Then $\vert \mathcal I^2_{\ell}(2) \vert$ is $5$-dimensional. We set 
$$ 
\mathbf P^5_{\ell} := \vert \mathcal I^2_{\ell}(2) \vert .
$$ 
This linear system parametrizes all the quadrics singular along $\ell$ and its general
member has rank $3$. Let $\ell \in G(1,4)$ and $Q \in \mathbf P^5_{\ell}$ be general. Then $Q$ defines a smooth curve $C := Q \cdot R$. 
 By the previous remarks $C$ is a general curve of genus $2$. Moreover, the ruling of planes of $Q$ cuts on $C$ a pencil $\vert \eta \vert$, where 
 $\eta \in \Pic^3C$ and  $\eta^{\otimes 2} \cong \mathcal O_C(1)$. Let $\mathbf P^{4*} = \vert \mathcal O_{\mathbf P^4}(1) \vert$. Then a general 
 $(Q,H) \in \mathbf P^5_{\ell} \times \mathbf P^{4*}$ defines a triple $(C, \eta, b)$ where $b := H \cdot C$, with  $(C, \eta)$ is constructed as above. 
This defines a rational map
$$
f_{\ell}: \mathbf P^5_{\ell} \times \mathbf P^{4*} \to \mathcal R_{2,6}
$$
sending $(Q,H)$ to the moduli point of $(C, \eta,b)$. Notice that $f_{\ell}$ is a map of varieties of the same dimension, depending of course on $\ell$. 
\par
Let $G_{\ell}\subset \Aut \mathbf P^4$ be the stabilizer of $R \cup \ell$. It is clear that $G_{\ell}$ acts linearly on the 
two factors of $\mathbf P^5_{\ell} \times \mathbf P^{4*}$. Then we consider its action on this product, sending $(Q,H)$ to $a(Q,H) := (a(Q), a(H))$. 
First we show that $f_{\ell}$ 
is birationally  the quotient map of this action.
\par  Assume $f_{\ell}(p_1) = f_{\ell}(p_2)$ for two general points $p_i = (Q_i, H_i)$, $i = 1, 2$. Let $C_i := Q_i \cdot R$ and let $\eta_i \in \Pic C^3_i$ 
be the line bundle defined by the ruling of $Q_i$. Then $C_1$, $C_2$ are biregular to the same general curve $C$ and $\eta_1$, $\eta_2$ are isomorphic.
This implies that there exists $a \in \Aut \mathbf P^4$ such that $a(C_1) = C_2$ and $a(Q_1) = Q_2$. We can assume that $\mathcal O_{C_i}(1)$ is 
general in $\Pic^6 C_i$ so that the stabilizer of $C_i$ in $\Aut \mathbf P^4$ is trivial. Let $b_i = C_i \cdot H_i$ then it follows $a(b_1) = b_2$ and $a(H_1) = H_2$. Moreover, notice that 
 the ruling of lines of $R$ is invariant by $a$, since $a(C_1) = C_2$. Hence $a(R) = R$. Since $a(Q_1) = a(Q_2)$ then $a(\ell) = \ell$. This implies that 
 $a \in G_{\ell}$ and that $(Q_2, H_2) = a(Q_1, H_2)$. 
 Since $\mathbf P^5 \times \mathbf P^{4*}$ and $\mathcal R_{2,6}$ are irreducible of the same dimension, it follows that $\mathbf P^5 \times 
 \mathbf P^{4*} / G_{\ell}$ is birational to $\mathcal R_{2,6}$. \par
Finally, we claim that $G_{\ell}$ is isomorphic to $\mathbb Z/2\mathbb Z$ for a general $\ell$.  Therefore
$\mathbf P^5 \times \mathbf P^{4*}$ descends to a $\mathbf P^5$-bundle over the quotient space $\mathbf P^4 / G_{\ell}$
and the latter is  rational.  
\end{proof}

\begin{proof} [Proof of the claim]
The stabilizer of $R$ in $\Aut \mathbf P^4$ is isomorphic to $\Aut R$. Since $R$ is the blow-up of $\mathbf P^2$ at a point $o$, 
$\Aut R$ is isomorphic to the  stabilizer of $o$ in $\Aut \mathbf P^2$. Notice that $\Aut R$ is $6$-dimensional and acts on the $6$-dimensional variety
 $G(1,4)$. Recall that we have a map
$$
\beta: R \to \vert \mathcal I_R(2) \vert
$$
such that $\beta(v)$ is the unique $Q_v \in \vert \mathcal I_R(2) \vert$ singular at $v$. By Proposition \ref{contraction} $\beta$ is the contraction of 
the exceptional  line $E \subset R$ and $o := \sigma(E)$ is the parameter point of the unique rank $3$ quadric $Q_E$ containing $R$. Let  
$$
r_{\ell}: \vert \mathcal I_R(2) \vert \to \vert \mathcal O_{\ell}(2) \vert
$$
be the restriction map. We define the non empty open sets:
\medskip \par
\begin{enumerate}
\item $U_1 := \lbrace \ell \in G(1,4) \ \vert \ r_{\ell} \ \text {\it is an isomorphism} \rbrace$,
\item $U_2 := \lbrace \ell \in G(1,4) \ \vert \ \text{ \it $\ell$ is transversal to $Q_E$} \rbrace$
\end{enumerate}
and prove our claim for $\ell \in U_1 \cap U2$. Let $\Delta_{\ell} := r_{\ell}^*\Delta$ be the pull back of the diagonal conic $\Delta := \lbrace 2x \mid \ x \in \ell 
\rbrace$. Then $\Delta_{\ell}$  is a conic in $\vert \mathcal I_R(2) \vert$, more precisely it is the family of quadrics  through $R$ which are tangent to $\ell$.  
Let $a \in G_{\ell}$ then $a(Q_E\cap \ell) = Q_E \cap \ell$. Since $Q_E$ is transversal to 
$\ell$ it follows that   $$
Q_E \cap \ell = \lbrace x_1 , x_2 \rbrace
$$
with $x_1 \neq x_2$. Therefore the parameter point $o$ of $Q_E$ is not in $\Delta_{\ell}$. It is also clear that the induced automorphism $a_*: \vert 
\mathcal I_R(2) \vert \to \vert \mathcal I_R(2) \vert$ leaves $\Delta_{\ell}$ invariant. Let $Q_v \in \vert \mathcal I_R(2) \vert$ be singular at $v$ then 
$Q_{a(v)} = a(Q_v)$ and this equality implies that $a_{|R}: R \to R$ is uniquely defined by $a_*$. Since $a_{|R}$ is the identity if and only if
$a$ is, it follows that $a$ is uniquely reconstructed  from $a_*$. \par
Now $a_*(o) = o$  implies $a_*(L) = L$ and $a_*(T_1 \cup T_2) = T_1 \cup T_2$, where $L$ is the polar line of $o$ with respect to $\Delta_{\ell}$ 
and $T_1,  T_2$ are the tangent lines to
$\Delta_{\ell}$ at the two points of $L \cap \Delta_{\ell}$. Let $P$ be the pencil of conics which are tangent to $\Delta_{\ell}$ at $L \cap \Delta_{\ell}$. 
Then $a_*$ acts on $P$ fixing each of the
elements $\Delta_{\ell}$, $T_1 + T_2$,  $2L$. Hence $a_*$ acts as the identity on $P$. To conclude just recall that the group of planar automorphisms
acting  on a pencil like $P$ as the
identity is isomorphic to $\mathbb Z/ 2\mathbb Z$. Indeed, it is generated by the involution leaving each conic of $P$ invariant, with $o$ and the points of 
$L$ fixed. Hence $G_{\ell}$ is generated by $a$, where $a_*$ is such an involution.
 \end{proof}


\end{document}